\documentclass[12pt]{amsart}

\usepackage{amsthm,color,amsmath,amssymb}

\begin{document}

\title[CR embeddings into low dimensional spheres]{On CR embeddings of strictly pseudoconvex hypersurfaces into spheres in low dimensions}
\author{Peter Ebenfelt}
\address{Department of Mathematics, University of California at San Diego, La Jolla, CA 92093-0112}
\email{pebenfel@math.ucsd.edu}
\author{Andr\'e Minor}
\address{Department of Mathematics, University of California at San Diego, La Jolla, CA 92093-0112}
\email{aminor@math.ucsd.edu}
\thanks{The authors were supported in part by DMS-1001322.}


\newtheorem{prop}{Proposition}
\newtheorem{lem}{Lemma}
\newtheorem{thm}{Theorem}
\newtheorem{cor}{Corollary}
\newtheorem{ex}{Example}
\newtheorem{defi}{Definition}
\numberwithin{equation}{section}
\numberwithin{lem}{section}
\numberwithin{prop}{section}
\numberwithin{ex}{section}

\newcommand{\C}{\mathbb C}
\newcommand{\R}{\mathbb R}
\newcommand{\bC}{\mathbb C}
\newcommand{\bR}{\mathbb R}
\newcommand{\bP}{\mathbb P}
\newcommand{\bS}{\mathbb S}
\newcommand{\db}{|\!|}

\newcommand \re{\text{Re}\,}
\newcommand \im{\text{Im}\,}
\newcommand \Rk{\text{Rk}\,}
\newcommand \tr{\text{tr}\,}
\newcommand \Aut{\text{Aut}\,}

\abstract It follows from the 2004 work of the first author, X.Huang, and D. Zaitsev that any local CR embedding $f$ of a strictly psedoconvex hypersurface $M^{2n+1}\subset\bC^{n+1}$ into the sphere $\bS^{2N+1}\subset \bC^{N+1}$ is rigid, i.e.\ any other such local embedding is obtained from $f$ by composition by an automorphism of the target sphere $\bS^{2N+1}$, {\it provided} that the codimension $N-n<n/2$. In this paper, we consider the limit case $N-n=n/2$ in the simplest situation where $n=2$, i.e.\ we consider local CR embeddings $f\colon M^5\to \bS^7$. We show that there are at most two different local embeddings, up to composition with an automorphism of $\bS^7$. We also identify a subclass of 5-dimensional, strictly pseudoconvex hypersurfaces $M^5$ in terms of their CR curvatures such that rigidity holds for local CR embeddings $f\colon M^5\to \bS^7$.
\endabstract

\maketitle

\section{Introduction}

In the theory of Levi nondegenerate (real) hypersurfaces $M=M^{2n+1}$ in $\bC^{n+1}$, the role of flat models is played by the nondegenerate hyperquadrics (\cite{CM74}). In particular, in the strictly pseudoconvex case, the flat model is the (unit) sphere $\bS=\bS^{2n+1}\subset \bC^{n+1}$, given by
$$
\sum_{k=1}^{n+1}|z_k|^2 = 1.
$$
In recent years, great effort has being dedicated to the study of existence and uniqueness/classification of transversal CR mappings sending a Levi nondegenerate hypersurface $M\subset \bC^{n+1}$ into the corresponding flat model in a higher dimensional complex space $\bC^{N+1}$. The existence problem seems very difficult and delicate in general, and will not be addressed here; the reader is referred to e.g. \cite{Webster78b}, \cite{Forstneric86}, \cite{Faran88}, \cite{Lempert90}, and \cite{KimOh09}, \cite{JPD11} for a sampling of results in this direction. In the case where $M$ is the sphere $\bS^{2n+1}$ itself, the existence problem is of course trivial, but the uniqueness/classification problem has attracted much attention. Here, uniqueness/classicication has to be understood in the appropriate way. Given any CR mapping $f\colon \bS^{2n+1}\to \bS^{2N+1}$, automorphisms $R$ and $T$ of $\bS^{2n+1}$ and $\bS^{2N+1}$, respectively, the mapping $g:=T\circ f\circ R$ is a seemingly "different" mapping, but contains no interesting new information since the automorphism group of the sphere is well understood. (The automorphism group of $\bS^{2N+1}$ consists of global automorphisms of complex projective space $\bP^{N+1}$, and is isomorphism to $SU(N+1,1)$.) The classification problem consists of classifying mappings up to the equivalence relation given by this left/right action of the automorphism groups of the target and source spheres. There is a vast literature on this topic. Mappings between spheres in general are not the main focus of this paper and, therefore, we only refer the reader to the fairly recent publications \cite{HuangJiXu06}, \cite{JDPLeblPeters07} for more details and references.

A CR mapping, $f_0\colon M^{2n+1} \to \bS^{2N_0+1}$, of a strictly pseudoconvex hypersurface $M\subset\bC^{n+1}$ into the sphere $\bS\subset\bC^{N_0+1}$ is said to be {\it rigid} in a some codimensional range $N-n<r$ if any other CR mapping $f\colon M^{2n+1} \to \bS^{2N+1}$ with $N-n<r$ is obtained as $f=T\circ L\circ f_0$, where $T$ is  some automorphism of the target sphere $\bS^{2N+1}$ and $L$ denotes the linear embedding of $\bS^{2N_0+1}$ into a subspace section of $\bS^{2N+1}$. We shall say that rigidity holds (in a given codimension range) if there exists a rigid mapping (for this range). If $M=\bS^{2n+1}$, then rigidity turns out to be equivalent to the statement that the obvious linear embedding is the only non-constant CR mapping up to the equivalence described in the previous paragraph. In this paper, we shall consider the situation where rigidity potentially breaks down for non-constant CR mappings $f\colon M^{2n+1} \to \bS^{2N+1}$.  As will be explained below, in the case where $M=\bS^{2n+1}$ this situation is understood due to the work of X. Huang and S. Ji (\cite{HuangJi01}). To date, however, nothing seems to be known in the corresponding situation for more general $M$. As a modest first step, we shall investigate the simplest case in this note. We shall now explain this more precisely.

Let $M$ be a smooth, connected, strictly pseudoconvex hypersurface in $\bC^{n+1}$. An invariant of $M$ that can be used to determine a range of codimensions ($N-n$) for which local CR embeddings into the sphere $\bS^{2N+1}\subset\bC^{N+1}$ are rigid is its {\it CR complexity} (introduced in \cite{ESh10}; see also \cite{BEH08})
\begin{equation}\label{CRcomplexity}
\mu(M):=\min\{N_0-n\colon \exists  f_0\colon M\to \bS^{2N_0+1}\subset \bC^{N_0+1}\ \text{{\rm with $f_0$ locally a CR embedding.}}\},
\end{equation}
where we have to allow $\mu(M)$ to potentially take the value $\infty$ (meaning that $M$ cannot be locally embedded into $\bS^{2N_0+1}\subset \bC^{N_0+1}$ for any $N_0$; see \cite{Forstneric86}, \cite{Faran88}, \cite{Lempert90}). We remark that a CR mapping of $M$ into the sphere (or any strictly pseudoconvex hypersurface) is a transversal local embedding if and only if it is non-constant and, hence, it suffices to consider $f_0$ that are non-constant in \eqref{CRcomplexity}. We note that the prototype for $M$ with $\mu(M)=0$ is the sphere (and any $M$ with $\mu(M)=0$ is locally spherical, i.e.\ locally biholomorphic to the sphere). An example of $M$ with $\mu(M)=1$ is given by any (non-spherical) real ellipsoid; see Section \ref{Examples} for further examples.

Our starting point is the following rigidity result due to the first author, X. Huang, and D. Zaitsev:

\begin{thm}\label{Rigid1}  {\rm (\cite{EHZ04}, \cite{EHZ05})} Let $M$ be a smooth, connected, strictly pseudoconvex hypersurface in $\bC^{n+1}$ with CR complexity $\mu(M)=N_0-n$. If $f\colon M\to \bS^{2N+1}\subset \bC^{N+1}$ (necessarily with $N\geq N_0$) is a non-constant CR mapping such that
\begin{equation}\label{codimension}
N-n+\mu(M)<n
\end{equation}
then $f=T\circ L\circ f_0$, where $f_0$ is any CR mapping in the definition \eqref{CRcomplexity} of $\mu(M)$, $L\colon \bC^{N_0+1}\to \bC^{N+1}$ is a linear embedding sending $\bS^{2N_0+1}$ into a subspace section of $\bS^{2N+1}$, and $T\colon \bP^{N+1}\to\bP^{N+1}$ is an automorphism of $\bS^{2N+1}$.
\end{thm}

We should point out that in the special case where $M$ is the sphere $\bS^{2n+1}$, so that $\mu(M)=0$, the codimensional restriction \eqref{codimension} reads $N-n<n$. In this case, Theorem \ref{Rigid1} was first proved by Faran \cite{Faran86} assuming that the mapping $f$ is real-analytic; the smoothness requirement on $f$ was subsequently lowered by several authors and the strongest result in this direction ($C^2$-smoothness suffices) is due to X. Huang \cite{Huang99}. Earlier versions of Theorem \ref{Rigid1} were proved by Webster \cite{Webster79} in the special case $N=n+1$, with $n\geq 2$ in the case where $M=\bS^{2n+1}$ and $n\geq 3$ (with no reference to the CR complexity) in the general case. In 1982, Faran \cite{Faran82} showed that Webster's condition $n\geq 2$ in the sphere case is sharp by showing that there are precisely 4 non-equivalent (non-constant) mappings sending $\bS^3\subset \bC^2$ into $\bS^5\subset \bC^{3}$ (i.e. $n=1$, $N=2$). One of these mappings (the Whitney map) generalizes to show that the condition $N-n<n$ in the spherical case ($\mu(M)=0$) in Theorem \ref{Rigid1} is also sharp: there are at least 2 non-equivalent maps (the linear map and the Whitney map) $\bS^{2n+1}\to \bS^{2N+1}$ when $N-n=n$. In 2001, X. Huang and S. Ji \cite{HuangJi01} showed that these are the only two mappings up to equivalence (in fact, up to composition with an automorphism of the target sphere). The purpose of this note is to extend the result in \cite{HuangJi01} to the more general setting of Theorem \ref{Rigid1} in the simplest possible case $n=2$. More precisely, we wish to consider the situation where $M$ is a strictly pseudoconvex hypersurface in $\bC^3$ ($n=2$) and $f\colon M\to \bS^{2N+1}\subset \bC^{N+1}$ is a non-constant CR mapping such that $N-n+\mu(M)=n$, i.e.\ $N-2+\mu(M)=2$ or, equivalently, $N+\mu(M)=4$. Since $\mu(M)=N_0-n\leq N-n$ by definition, there are only two cases: (1) $\mu(M)=0$ and $N=4$, in which case the result in \cite{HuangJi01} asserts that there are precisely 2 mappings up to composition with an automorphism of the target sphere, and (2) $\mu (M)=1$ and $N=3$. In this note, we shall show that also in the case (2) there are at most 2 mappings up to composition with an automorphism of the target sphere:

\begin{thm}\label{Main} Let $M$ be a smooth, connected, strictly pseudoconvex hypersurface in $\bC^{3}$. Then there are at most two non-constant CR mappings into $\bS^{7}\subset \bC^4$ up to composition with an automorphism of $\bS^7$; i.e.\ if $f_1\colon M\to \bS^{7}\subset \bC^{4}$, $f_2\colon M\to \bS^{7}\subset \bC^{4}$, and $f_3\colon M\to \bS^{7}\subset \bC^{4}$ are non-constant CR mappings, then there is an automorphism $T$ of $\bS^7$ and $j\neq k$ with $j,k\in \{1,2,3\}$ such that $f_k=T\circ f_j$.
\end{thm}

By combining Theorem \ref{Rigid1}, the result in \cite{HuangJi01}, and the main new result in this paper, we may summarize the situation for $n=2$ as follows:

\begin{cor}\label{Main0} Let $M$ be a smooth, connected, strictly pseudoconvex hypersurface in $\bC^{3}$ with CR complexity $\mu(M)\leq 1$. Then, for
\begin{equation}\label{codimension3/2}
N+\mu(M)< 4,
\end{equation}
there is only one non-constant CR mapping $f\colon M\to \bS^{2N+1}\subset \bC^{N+1}$ up to composition with an automorphism of the target sphere $\bS^{2N+1}$, and, for
\begin{equation}\label{codimension2}
N+\mu(M)= 4,
\end{equation}
there are at most two non-constant CR mappings $f\colon M\to \bS^{2N+1}\subset \bC^{N+1}$ up to composition with an automorphism of the target sphere $\bS^{2N+1}$.
\end{cor}

We shall in fact prove a more refined version of Theorem \ref{Main} in which CR curvature invariants of $M$ divide the hypersurfaces in Theorem \ref{Main} into two classes, one for which there is at most one mapping and one for which there are at most two (again, of course, up to composition with an automorphism with the target sphere). In order to formulate this result, we need to introduce some notation.

Let $M\subset \bC^{n+1}$ be a smooth, connected, strictly pseudoconvex hypersurface and $p\in M$. The CR (or pseudoconformal) curvature (the CR analogue of the Weyl curvature in Riemannian geometry) was introduced by Chern and Moser in \cite{CM74}. Let us choose a contact form $\theta$ on $M$ near $p$ such that the Levi form of $M$ is positive definite, and extend $\theta$  to an admissible (in the sense of Webster \cite{Webster78}) CR coframe $(\theta,\theta^1,\ldots,\theta^n, \theta^{\bar 1},\ldots,\theta^{\bar n})$, so that
\begin{equation}\label{Leviform}
d\theta=ig_{\alpha\bar\beta}\theta^\alpha\wedge\theta^{\bar\beta},
\end{equation}
where the matrix $(g_{\alpha\bar\beta})$ represents the Levi form of $M$; here and in what follows, we use the summation convention and let $\alpha,\beta,\ldots$ run over the index set $\{1,\ldots, n\}$.
Relative to this coframe, the CR curvature is then represented, near $p$, by a Hermitian curvature tensor (field)
\begin{equation}\label{CRcurvature}
S\colon T^{1,0}\times T^{1,0}M\times T^{1,0}\times T^{1,0}M \to \bC,
\end{equation}
where $T^{1,0}M$ denotes the bundle of $(1,0)$ vector fields on $M$. Equivalently, using the dual frame on $T^{1,0}M$, we can represent $S$ in tensor notation as $S=(S_{\alpha\bar\beta\nu\bar\mu})$. If we let $X$ denote the space of symmetric $(2,0)$ tensors on $M$ near $p$,
\begin{equation}\label{X}
X:=\{x_{\beta\mu}\theta^\beta\otimes\theta^\mu\colon x_{\beta\mu}=x_{\mu\beta}\},
\end{equation}
then $S$ defines a (pointwise) linear map $L_S\colon X\to X$ via (cf.\ \cite{KimOh09})
\begin{equation}\label{L-S}
L_S\colon x=(x_{\beta\mu})\mapsto (S_{\alpha}{}^{\beta}{}_{\nu}{}^{\mu}x_{\beta\mu}),
\end{equation}
where we use the Levi form $(g_{\alpha\bar\beta})$ and its inverse $(g^{\alpha\bar\beta})$ to lower and raise indices. The CR curvature in \eqref{CRcurvature}, as well as the Levi form, and consequently the linear map $L_S$ depend on the contact form $\theta$, but the rank and the sign of the trace of $L_S$ are CR invariants (as long as we require the Levi form with respect to $\theta$ to be positive definite); we shall use the notation $\Rk L_S$ and $\tr L_S$ for the rank and trace of $L_S$, respectively. We point out that $M$ is locally spherical (i.e. $\mu(M)=0$) if and only if the CR curvature $S$ is identically 0, which of course happens if and only if $\Rk(L_S)$ is identically 0. We should also point out that $\dim X=n(n+1)/2$; thus, $\dim X=3$ when $n=2$.

We shall prove the following result, of which part (i) is precisely Theorem \ref{Main}:

\begin{thm}\label{Main1} Let $M$ be a smooth, connected, strictly pseudoconvex hypersurface in $\bC^{3}$. Then the following hold:
\medskip

\noindent {\rm{\bf (i)}} There are at most two non-constant CR mappings into $\bS^{7}\subset \bC^4$ up to composition with an automorphism of $\bS^7$; i.e. if $f_1\colon M\to \bS^{7}$, $f_2\colon M\to \bS^{7}$, and $f_3\colon M\to \bS^{7}$ are non-constant CR mappings, then there are an automorphism $T$ of $\bS^7$ and $j\neq k$ with $j,k\in \{1,2,3\}$ such that $f_k=T\circ f_j$.
\smallskip

\noindent {\rm{\bf (ii)}} If there exists an open subset $U\subset M$ such that $\Rk L_S<2$ in $U$, or $\Rk L_S= 2$ and $\tr L_S < 0$ in $U$, then there is at most one non-constant CR mapping into $\bS^{7}\subset \bC^4$ up to composition with an automorphism of $\bS^7$; i.e.\ if $f_1\colon M\to \bS^{7}$ and $f_2\colon M\to \bS^{7}$ are non-constant CR mappings, then there is an automorphism $T$ of $\bS^7$ such that $f_2=T\circ f_1$.
\end{thm}

A crucial ingredient in the proof of Theorem \ref{Main1} is a careful study of the Gauss equation, which for a transversal local CR embedding $f\colon M\to \bS^{2N+1}$ relates the CR curvature of $M$ to the CR second fundamental form of $f$. In proving rigidity results of the form given e.g.\ by Theorem \ref{Rigid1}, the corresponding step is typically dealt with in the following way: Assume there are two solutions to the Gauss equation. By subtracting the two Gauss equations, the curvature term is eliminated and, under suitable conditions, it is possible to show that the two solutions are equal (up to an inevitable, but natural action). This approach will not work if more than one solution is expected, unless the expected solutions are known (as in the case $M=S^{2n+1}$ and $N-n=n$, where the linear embedding and the Whitney map are both solutions) and an arbitrary solution is compared to one of the known solutions. A novelty in the present paper is that we study the number of solutions to the Gauss equation directly, including the curvature term. As a result, we are also able to identify a new class of strictly pseudoconvex hypersurfaces in $\bC^3$, in terms of their CR curvatures, whose embeddings into the sphere remain rigid in a larger codimensional range than that predicted by Theorem \ref{Rigid1}. This is part (ii) of Theorem \ref{Main1}.

\section{Basic Setup}
We shall use the setup from \cite{EHZ04} (see also \cite{BEH08} and \cite{ESh10}), to which the reader is referred for details. Let $M$ be a smooth, connected, strictly pseudoconvex hypersurface in $\bC^{n+1}$, so that $M$ is a CR manifold of CR dimension $n$ and real dimension $2n+1$.
We will work with $M$, locally near a given point $p\in M$, by fixing a choice of pseudohermitian structure (i.e.\ a contact form) $\theta \in T^c(M)^{\bot}$, where $T^c(M)$ is the complex tangent bundle of the CR manifold M. As in the introduction, we extend $\theta$ to an adapted CR coframe for the complexified tangent space of $M$, $(\theta, \theta^\alpha, \theta^{\bar{\beta}})_{\alpha,\beta=1}^n$; i.e.\  $(\theta, \theta^\alpha)$ is a coframe for the space of $(1,0)$ forms (a.k.a.\ sections of $(T^*)^{1,0}M:=(T^{0,1}M)^\perp$), $\theta^{\bar\beta}=\overline{\theta^\beta}$, and \eqref{Leviform} holds. The matrix $(g_{\alpha\bar\beta})_{\alpha,\beta=1}^n$ is a representative of the Levi form. For the purposes of this paper we shall require, as we may, that $g_{\alpha \bar{\beta}} = \delta_{\alpha \bar{\beta}}$. By a slight abuse of notation, we shall frequently refer to the collection $(\theta, \theta^\alpha)_{\alpha=1}^n$, which determines the full coframe, as the admissible coframe. For a fixed pseudohermitian structure $\theta$ and choice of Levi form, the 1-forms $\theta^\alpha$ in the coframe are determined up to unitary transformations.

Let $f:M \rightarrow \hat{M} $ be a transversal, local CR embedding of a strictly pseudoconvex, hypersurface $M\subset \mathbb{C}^{n+1}$ into another such $\hat{M}\subset \mathbb{C}^{N+1}$. Due to the strict pseudoconvexity of $M$ and $\hat M$, as mentioned in the introduction, $f$ is transversal and a local embedding if and only if $f$ is non-constant. By Corollary 4.2 in \cite{EHZ04}, given any admissible coframe $(\theta, \theta^\alpha )_{\alpha=1}^{n}$ as above, defined locally near $p\in M$ there exists an admissible coframe $( \hat{\theta}, \hat{\theta}^A )_{A=1}^{N}$ on $\hat{M}$, defined near $f(p) \in \hat{M}$, so that:
\begin{equation}
f^{\ast}(\hat{\theta},\: \hat{\theta}^\alpha, \: \hat{\theta}^a)=(\theta,\: \theta^\alpha,\: 0),
\label{adaptedcond}
\end{equation}
and such that $\hat g_{A\bar B}=\delta_{A\bar B}$. Here, and for the rest of this paper, we use the following convention: Lower case Greek letters vary from 1 to n, $\alpha, \beta,\ldots \in \{1,...,n\}$, capital Roman letters vary from 1 to $N$, $A,B\ldots \in \{1,...,N \}$, and lower case Roman letters vary in the normal direction $a,b,\ldots \in \{n+1,...,N \} $. We shall say that the coframe $( \hat{\theta}, \hat{\theta}^A )$ is adapted to $f$ with respect to $(\theta, \theta^\alpha )$ if it satisfies \eqref{adaptedcond} and the Levi form in this coframe is the identity matrix.

Choose an adapted coframe on $\hat{M}$. We will omit the "hats" from the adapted coframe from now on. The CR second fundamental form of $f$ is given by
\begin{equation}
\Pi_f = \omega_{\alpha \:\:\: \beta}^{\:\:\:a}\, \theta^\alpha \otimes \theta^\beta \otimes L_a,
\label{2fftensor}
\end{equation}
where $( T, L_A, L_{\bar B} )$ is a dual frame to the coframe $( \theta, \theta^A, \theta^{\bar B})$, and the $\omega_{\alpha \: \: \: \beta}^{\:\:\:a}$ are determined by
\begin{equation}
f^{\ast}\hat{\omega}_{\alpha}^{\:\:\:a} = \omega_{\alpha \: \: \: \beta}^{\:\:\:a} \theta^\beta.
\label{2ffcond}
\end{equation}
The coefficients $\omega_{\alpha}{}^a{}_{\beta}$ are symmetric in the lower indices:
$$
\omega_{\alpha \: \: \: \beta}^{\:\:\:a}=\omega_{\beta \: \: \: \alpha}^{\:\:\:a}
$$
The second fundamental form is said to be {\it nondegenerate} if the collection of vectors $(\omega_{\alpha}{}^{a}{}_\beta)_{a=n+1}^{N}\in \mathbb{C}^{N-n}$ spans $\mathbb{C}^{N-n}$ as we vary $\alpha$ and $\beta$ over their index set.

In the paper by Chern and Moser \cite{CM74}, a complete system of invariant forms (a parallelism)  $$(\omega, \omega^\alpha,\omega^{\bar\beta}, \phi, \phi_\beta^{\: \: \: \alpha}, \phi_{\bar{\beta}}^{\: \: \: \bar{\alpha}}, \phi^{\alpha}, \phi^{\bar{\alpha}}, \psi )$$ is constructed on a certain principle G-bundle over $M$. By choosing an admissible CR coframe we may pull these forms back to $M$ (with $(\omega,\omega^\alpha,\phi)$ pulling back to $(\theta,\theta^\alpha,0)$). Using an adapted coframe on $\hat M$, the same can be done for $\hat M$ and these forms can then be pulled back to $M$ via the CR mapping $f$. The following result from \cite{EHZ04} (Theorem 6.1 there) will be important for the conclusion of this paper.

\begin{thm} {\rm (\cite{EHZ04})}
Let $f:M \rightarrow \mathbb{S}$ be a CR-embedding of a strictly pseudoconvex CR-manifold M of dimension $2n+1$, $n \geq 2$, into the unit sphere $\mathbb{S} \subset \mathbb{C}^{N+1}$. Let $(\omega_{\alpha \: \: \: \beta}^{\:\:\:a})$ be the second fundamental form of the embedding f with respect to some coframe on $\hat{M}=\mathbb{S}$ adapted to the embedding f and an admissible CR coframe on $M$. If the second fundamental form is nondegenerate, then any other embedding $\tilde{f}:M \rightarrow \mathbb{S}$ having the same second fundamental form relative to some (other) coframe on $\mathbb{S}$ adapted to $\tilde{f}$ and the same coframe on $M$ (i.e.\ if $(\omega_{\alpha \: \: \: \beta}^{\:\:\:a})=(\tilde{\omega}_{\alpha \: \: \: \beta}^{\:\:\:a})$ and $\tilde{f}^{\ast}(\tilde{\theta},\tilde{\theta}^\alpha)=f^{\ast}\big(\hat{\theta},
\hat{\theta}^\alpha)$ ), then
\begin{equation}
\tilde{f}^{\ast}\bigg(\tilde{\hat{\phi}}_B^{\: \: \: A}, \tilde{\hat{\phi}}^{ A}, \tilde{\hat{\psi}}\bigg) = f^{\ast} \bigg(\hat{\phi}_B^{\: \: \: A}, \hat{\phi}^{A}, \hat{\psi}\bigg).
\end{equation}
\label{ehz61}
\end{thm}

\section{The CR Curvature Tensor}\label{Curv-sect}
Given an admissible CR coframe $(\theta,\theta^\alpha)$ on $M$, the CR (pseudoconformal) curvature tensor of $M$,
\begin{equation}\label{CRcurvature2}
S\colon T^{1,0}\times T^{1,0}M\times T^{1,0}\times T^{1,0}M \to \bC,
\end{equation}
is represented by a Hermitian, traceless, curvature tensor (field) $S = (S_{\alpha \bar{\beta} \nu \bar{\mu}})$ on $M$; i.e.\ the coefficients $S_{\alpha \bar{\beta} \nu \bar{\mu}}$ satisfy:
\begin{equation}
S_{\alpha \bar{\beta} \nu \bar{\mu}} = S_{\nu \bar{\beta} \alpha \bar{\mu}}=S_{\nu \bar{\mu} \alpha \bar{\beta} } =S_{\bar{\beta} \alpha \bar{\mu}  \nu } (:=\overline{S_{{\beta} \bar\alpha {\mu}  \bar\nu }}),
\label{symcond}
\end{equation}
and the trace condition
\begin{equation}
S_{\mu \: \: \alpha \bar{\beta}}^{\:\: \mu} = 0,
\label{tracecond}
\end{equation}
where, as mentioned in the introduction, we use the summation convention, and use the Levi form and its inverse to lower and raise indices. The sectional curvature is represented by the degree 4 homogeneous polynomial $S(\zeta, \bar{\zeta})$, for $\zeta \in \mathbb{C}^n$, defined as
\begin{equation}
S(\zeta, \bar{\zeta}) = S_{\alpha \bar{\beta} \mu \bar{\nu}} \zeta^\alpha \zeta^\mu \zeta^{\bar{\beta}} \zeta^{\bar{\nu}}.
\label{Spoly}
\end{equation}
Notice that the trace condition \eqref{tracecond} is equivalent to $\Delta S(\zeta,\bar{\zeta})=0$, where $\Delta:=(\partial_\alpha\partial^\alpha)/4$, with $\partial_\alpha:=\partial/\partial \zeta^\alpha$, is the Laplacian.

For the remainder of this paper we focus on the case $n=2$. Using the symmetry properties in \eqref{symcond} for the coefficients of the curvature tensor, we see that the polynomial in \eqref{Spoly} can be written as:
\begin{equation}
S(\zeta, \bar{\zeta}) = Z\left(
\begin{array}{ccc} S_{1\bar{1} 1 \bar{1}}& 2S_{1\bar{1} 1 \bar{2}} & S_{1\bar{2} 1 \bar{2}}\\
2S_{1\bar{1} 2 \bar{1}}& 4S_{1\bar{1} 2 \bar{2}}&2S_{1\bar{2} 2 \bar{2}}\\
S_{2\bar{1} 2 \bar{1}}&2S_{2\bar{1} 2 \bar{2}}&S_{2\bar{2} 2 \bar{2}}
\end{array} \right)
Z^\ast
\end{equation}
where $Z$ is the row vector of degree 2 monomials in $\zeta$; $Z=((\zeta^1)^2, \zeta^1\zeta^2,(\zeta^2)^2)$. Using the Hermitian condition in \eqref{symcond} and the trace condition \eqref{tracecond} we also note that the sectional curvature is of the form:
\begin{equation}\label{sectionalS}
S(\zeta, \bar{\zeta}) = Z\left(
\begin{array}{ccc} a& b & c\\
\bar{b}& -4a&-b\\
\bar{c}&-\bar{b}&a
\end{array} \right)
Z^\ast
\end{equation}
where $a := S_{1\bar{1} 1 \bar{1}}\in \bR$, $b:=2S_{1\bar{1} 1 \bar{2}}\in \bC$, and $c:=S_{1\bar{2} 1 \bar{2}}\in \bC$. For a fixed contact form $\theta$, the coefficients $S_{\alpha \bar{\beta} \mu \bar{\nu}}$ are determined up to a unitary transformation in $\zeta$ (or equivalently a unitary change of the $\theta^\alpha$). If we let $\tilde{\zeta} = \zeta U$ for some $U\in SU(2,\bC)$, i.e.\ $U$ of the form
\[
U= \left( \begin{array}{cc}p&q\\-\bar{q}&\bar{p} \end{array} \right)
\]
with $|p|^2+|q|^2=1$, then by substituting into \eqref{Spoly} and computing, we observe that the new matrix elements in \eqref{sectionalS} are given by
\begin{equation}
\begin{array}{lll}
\tilde{a}= (|p|^4 -4|pq|^2+|q|^4)a+pq(|q|^2-|p|^2)b+\bar{p}\bar{q}(|q|^2-|p|^2)\bar{b}+p^2q^2c+\bar{p}^2\bar{q}^2\bar{c}\\
\tilde{b}= 6p\bar{q}(|p|^2-|q|^2)a+p^2(|p|^2-3|q|^2)b+\bar{q}^2(|q|^2-3|p|^2)\bar{b}-2p^3qc+2\bar{p}\bar{q}^3\bar{c}\\
\tilde{c}= 6p^2\bar{q}^2a+2p^3\bar{q}b-2p\bar{q}^3\bar{b}+p^4c+\bar{q}^4\bar{c}.
\end{array}
\end{equation}
Notice that $\tilde{c}$ is a holomorphic, degree 4 homogeneous polynomial in the complex variables $p$ and $\bar{q}$. Since this polynomial vanishes at $(p,\bar q)=(0,0)$, its zero set meets the unit sphere $|p|^2+|\bar q|^2=1$. (Indeed, it meets the unit sphere along a "bouquet" of at most 4 circles.) Thus, after a unitary transformation in $\zeta$, we may assume that $c=0$ and the curvature tensor S is represented by:
\begin{equation}
S(\zeta, \bar{\zeta}) = Z\left(
\begin{array}{ccc} a& b & 0\\
\bar{b}& -4a&-b\\
0&-\bar{b}&a
\end{array} \right)
Z^\ast.
\label{abccond}
\end{equation}
If desired, one can further normalize the coefficient matrix of $S(\zeta,\bar \zeta)$ so that the coefficient $b$ becomes real and non-negative, which would leave only a finite number of $U\in SU(2,C)$ preserving this normalization in the generic case; however, this will not be needed in this paper and will not be pursued here.

Let us now turn to the induced linear map $L_S\colon X\to X$, given by \eqref{L-S}, using the notation introduced above. Recall that the space $X$ of symmetric $(2,0)$ tensors, given by \eqref{X}, has dimension $3$ in the situation studied here ($n=2$). If we identify $X$ with the space of complex, symmetric $2\times2$ matrices in the obvious way, and choose the following basis for this space,
\begin{equation}\label{Xbasis}
e_1:=\left(
\begin{array}{cc} 1& 0\\
0&0
\end{array} \right),\quad
e_2:=\left(
\begin{array}{cc} 0& 1\\
1&0
\end{array} \right),\quad
e_3:=\left(
\begin{array}{cc} 0& 0\\
0&1
\end{array} \right),
\end{equation}
then a straightforward calculation (left to the reader) reveals that in this basis $L_S$ is represented by the following matrix:
\begin{equation}\label{L-Smatrix}
L_S=\left(
\begin{array}{ccc} a& b & 0\\
b/2& -2a&-\bar b/2\\
0&-\bar{b}&0
\end{array} \right),
\end{equation}
where we have used that fact that $c$ has been normalized to be zero. We readily obtain the following:

\begin{prop}\label{L-Sprop} In an admissible CR coframe $(\theta,\theta^\alpha)$ for $M\subset \bC^3$ such that the sectional CR curvature has the normalized form \eqref{abccond}, the following holds:
\medskip

\noindent{\rm{\bf (i)}} $\tr L_S=-a$.
\smallskip

\noindent{\rm{\bf (ii)}} $a\neq 0$ and $b\neq 0$ $\iff$ $\Rk L_S=3$.
\smallskip

\noindent{\rm{\bf (iii)}} $a\neq 0$ and $b=0$ $\iff$ $\Rk L_S=2$ and $\tr L_S\neq 0$.
\smallskip

\noindent{\rm{\bf (iv)}} $a= 0$ and $b\neq 0$ $\iff$ $\Rk L_S=2$ and $\tr L_S=0$.
\smallskip

\noindent{\rm{\bf (v)}} $a= 0$ and $b=0$ $\iff$ $\Rk L_S<2$.
\end{prop}

\section{The Gauss Equation}\label{Gauss-sect}
Suppose that $f:M\subset \mathbb{C}^{n+1} \rightarrow \mathbb{S} \subset \mathbb{C}^{N+1}$ is a transversal CR embedding (locally) of a strictly pseudoconvex CR hypersurface M into the sphere, and let us choose admissible and adapted coframes with respect to $f$ as described above. In \cite{EHZ04} the following CR Gauss Equation is established:
\begin{align*}
S_{\alpha \bar{\beta} \mu \bar{\nu}} = & - g_{a\bar{b}}\omega_{\alpha \: \: \: \mu}^{\:\:\:a}\omega_{\bar{\beta} \: \: \: \bar{\nu}}^{\:\:\:\bar{b}}+\frac{1}{n+2}(\omega_{\gamma \: \: \: \alpha}^{\:\:\:a}\omega_{\: \: \: a \bar{\beta}}^{\gamma}g_{\mu \bar{\nu}}+\omega_{\gamma \: \: \: \mu}^{\:\:\:a}\omega_{\: \: \: a \bar{\beta}}^{\gamma}g_{\alpha \bar{\nu}}   + \omega_{\gamma \: \: \: \alpha}^{\:\:\:a}\omega_{\: \: \: a \bar{\nu}}^{\gamma}g_{\mu \bar{\beta}}\\ &+ \omega_{\gamma \: \: \: \mu}^{\:\:\:a}\omega_{\: \: \: a \bar{\nu}}^{\gamma}g_{\alpha \bar{\beta}} )- \frac{\omega_{\gamma \: \: \: \delta}^{\:\:\:a}\omega_{\:\:\:a}^{\gamma \: \: \: \delta}}{(n+1)(n+2)}(g_{\alpha \bar{\beta}}g_{\mu \bar{\nu}}+g_{\alpha \bar{\nu}}g_{\mu \bar{\beta}}).
\end{align*}
Using the fact that the Levi form has been normalized to be the identity matrix, we obtain for the sectional CR curvature:
\begin{align*}
S(\zeta,\bar{\zeta}) &= & -\sum_{a}\omega_{\alpha \: \: \: \mu}^{\:\:\:a}\omega_{\bar{\beta}\: \: \: \bar{\nu}}^{\:\:\:\bar{a}}\zeta^{\alpha}\zeta^{\mu}\zeta^{\bar{\beta}}\zeta^{\bar{\nu}}+
\frac{4}{n+2}\omega_{\gamma \: \: \: \alpha}^{\:\:\:a}\omega_{\: \: \: a \bar{\beta}}^{\gamma} \zeta^{\alpha}\zeta^{\bar{\beta}}|\zeta|^2-\frac{2\omega_{\gamma \: \: \: \delta}^{\:\:\:a}\omega_{\:\:\:a}^{\gamma \: \: \: \delta}}{(n+1)(n+2)}|\zeta|^4 \\
&=& -\sum_{a}|\omega^a(\zeta)|^2 +\frac{1}{n+2}\left|\!\left|\frac{\partial \omega(\zeta)}{\partial\zeta}\right|\!\right|^2 \db\zeta\db^2-\frac{1}{2(n+1)(n+2)}\left|\!\left|\frac{\partial^2\omega(\zeta)}{\partial \zeta^2}\right|\!\right|^2\db\zeta\db^4
\end{align*}
where
\begin{equation}\label{SFF+}
\omega^a(\zeta):= \omega_{\alpha\:\:\:\beta}^{\:\:\:a}\zeta^\alpha \zeta^\beta
,\quad \db\zeta\db^2:=\sum_\alpha|\zeta^\alpha|^2.
\end{equation}
Here, we have also used the notation
$$
\left|\!\left|\frac{\partial \omega(\zeta)}{\partial\zeta}\right|\!\right|^2,\quad
\left|\!\left|\frac{\partial^2\omega(\zeta)}{\partial \zeta^2}\right|\!\right|^2
$$
for suitable norms (which will not be important in this paper) of the corresponding derivative tensors. For our purposes, we note that the Gauss Equation asserts that there exists an $n\times n$ Hermitian matrix $A$ such that:
\begin{equation}
S(\zeta,\bar{\zeta}) + \zeta A \zeta^\ast \db\zeta\db^2 = -\sum_{a}|\omega^a(\zeta)|^2.
\label{Acond}
\end{equation}
The Gauss Equation also specifies the matrix $A=(A_{\alpha \bar{\beta}})$ to be
\begin{equation}
A_{\alpha \bar{\beta}}= \sum_{a} \left(-\frac{4}{n+2} \omega_{\gamma \:\:\: \alpha}^{\:\:\:a}\omega_{ \:\:\: a \bar{\beta}}^{\gamma}+2\frac{\omega_{\gamma \:\:\: \delta}^{\:\:\:a}\omega_{ \:\:\: a }^{\gamma \:\:\: \delta}}{(n+1)(n+2)}\delta_{\alpha \bar{\beta}}\right),
\label{Acomponents}
\end{equation}
although this will not be important for our purposes in this paper.

We begin with the following observation.
\begin{lem}\label{Fischer}
Suppose equation \eqref{Acond} is satisfied with some harmonic polynomial $S(\zeta,\bar\zeta)$, some Hermitian $n\times n$ matrix $A$,  and some holomorphic polynomials $\omega^a(\zeta)$. Then, $\omega^a(\zeta)\equiv 0$, for all $a$, if and only if $S(\zeta,\bar\zeta)=0$.
\end{lem}
\begin{proof} The "if" part is obvious due to the negative sign on the right in \eqref{Acond}. The converse is a direct consequence of a classical result due to E. Fischer (see e.g.\ \cite{Shapiro89}): Every polynomial can be uniquely decomposed as a sum of a harmonic polynomial and a multiple of $\db\zeta\db^2$.
\end{proof}

We now consider a CR embedding $f:M \rightarrow \mathbb{S}^7$ of a strictly pseudoconvex CR hypersurface $M\subset \mathbb{C}^3$ into the sphere $\mathbb{S}^7 \subset \mathbb{C}^4$. In this situation, we have $n=2$, $N=3$, and $N-n=1$. The CR second fundamental form of $f$ satisfies the Gauss Equation \eqref{Acond}, where the sum on the right has only one term ($a=3$). By Lemma \ref{Fischer}, the second fundamental form will be identically zero on an open subset $U\subset M$ if and only if the CR curvature is identically zero on $U$, or equivalently, $M$ is locally spherical on $U$. It follows from Theorem \ref{Rigid1} in this case  that $f$ equals $T\circ L\circ f_0$ on an open subset $V\subset U\subset M$, where $f_0$ is a local CR diffeomeorphism of $V$ onto an open piece of the sphere $\bS^{5}\subset \bC^{3}$ ($\mu(V)=0$ and $N-n=1<2=n$). Thus, if we have two non-constant CR mappings $f_1\colon M\to \bS^{7}$ and $f_2\colon M\to \bS^7$, then, after composition with yet another automorphism of the target sphere if necessary, they will agree on an open subset of $M$ and, therefore, will agree on all of $M$ by unique continuation of CR mappings along the connected, strictly pseudoconvex (or simply minimal) hypersurface $M$. Thus, the conclusion in Theorem \ref {Main1} (ii) in the case $\Rk L_S=0$ on $U$ follows. Moreover, by Proposition \ref{L-Sprop}, we have $\Rk L_S<2$ if and only if $\Rk L_S=0$ and hence, we conclude that the conclusion of Theorem \ref {Main1} (ii) holds if $\Rk L_S<2$ on an open subset $U\subset M$.

In view of the previous paragraph, in what follows we may assume that we are at a point where the CR curvature $S$ and the second fundamental form of $f$ do not vanish; i.e.\ we may assume that $ab\neq 0$ and the single component $\omega^3(\zeta)\not\equiv 0$. Now, given a Hermitian $2\times 2$ matrix
\begin{equation}\label{Amatrix}
A= \left( \begin{array}{cc} \tau & \eta \\ \bar{\eta} & \rho \end{array} \right),\quad \tau,\rho\in \bR,\ \eta\in\bC,
\end{equation}
a straightforward calculation shows that we may rewrite $\zeta A\zeta^*\db\zeta\db^2$ in the form $Z\tilde A Z^*$, where $\tilde A$ is the Hermitian $3\times 3$ matrix given by
$$
\tilde A:=\left(
\begin{array}{ccc} \tau & \sigma & 0\\
\bar{\sigma}& \tau + \rho &\sigma\\
0&\bar{\sigma}&\rho
\end{array}
\right).
$$
Thus, by combining equations \eqref{abccond} and \eqref{Acond}, we conclude that there exists a $2\times 2$ Hermitian matrix $A$ of the form \eqref{Amatrix} associated to the embedding $f$ such that:
\begin{equation}\label{temp1}
Z
\left[
\left(
\begin{array}{ccc} a& b & 0\\
\bar{b}& -4a&-b\\
0&-\bar{b}&a
\end{array} \right)
+
\left(
\begin{array}{ccc} \tau & \sigma & 0\\
\bar{\sigma}& \tau + \rho &\sigma\\
0&\bar{\sigma}&\rho
\end{array}
\right)
\right]
Z^\ast
= -|\omega(\zeta)|^2
\end{equation}
where we use the notation $\omega(\zeta)$ for the single component $\omega^3(\zeta)$, defined in \eqref{SFF+} with $a=3$, and $Z$, as above, is the row vector $Z=((\zeta^1)^2,\zeta^1\zeta^2,(\zeta^2)^2)$.
Let us introduce the notation $T_A$ for the sum of the two $3\times3$ matrices appearing in \eqref{temp1},
\begin{equation}
T_A := \left(
\begin{array}{ccc} \tau + a&\sigma+ b & 0\\
\bar{\sigma}+\bar{b}&\tau+\rho -4a&\sigma-b\\
0&\bar{\sigma}-\bar{b}&\rho +a
\end{array} \right).
\label{ta}
\end{equation}
Thus, we conclude that if $M$ has a (local) CR embedding into the sphere $\mathbb{S}^7$ then there exists a Hermitian $2 \times 2$ matrix $A$ such that
\begin{equation}
Z T_A Z^\ast = -|\omega(\zeta)|^2
\label{tacond}
\end{equation}
where $\omega(\zeta)$, a degree 2 homogeneous polynomial in $\zeta$, represents the second fundamental form of $f$.

Now, recall that the curvature tensor S (i.e.\ the coefficients $a,b$) is fixed. We will show that there are at most 2 distinct matrices A, with corresponding degree 2 homogeneous polynomials $\omega(\zeta)$, which satisfy equation \eqref{tacond}. It is easy to see that if two pairs $(A,\omega(\zeta))$ and $(\tilde A,\tilde \omega(\zeta))$ satisfy \eqref{tacond}, then $A=\tilde A$ if and only if $|\omega(\zeta)|^2\equiv |\tilde \omega(\zeta)|^2$ (or, equivalently, $\omega(\zeta)\equiv e^{it}\tilde\omega(\zeta)$ for some $t\in \bR$). Thus, counting matrices $A$ that solve \eqref{tacond} for some $\omega(\zeta)$ is equivalent to counting (potential) second fundamental forms $\omega(\zeta)$, up to the equivalence relation
\begin{equation}\label{SFFeq}
\omega(\zeta)\sim \tilde\omega(\zeta) \iff \exists t\in\bR\colon \omega(\zeta)\equiv e^{it}\tilde\omega(\zeta),
\end{equation}
that satisfy \eqref{tacond} for some $A$. Using this, we will then conclude the proof of Theorem \ref{Main1} in the next section. We begin with the following:
\begin{lem}
Given a $2 \times 2$ Hermitian matrix A, the equation \eqref{tacond} is satisfied for some homogeneous degree $2$ polynomial $\omega(\zeta)\not\equiv 0$ if and only if $\Rk(T_A)=1$ and $T_A \leq 0$.
\label{Alem1}
\end{lem}
\begin{proof}
Suppose \eqref{tacond} is satisfied with $\omega(\zeta)\not\equiv 0$. Observe that we may express $\omega(\zeta)$ in the form $Zv$, where $Z$ is as above and $v\neq 0$ is a column vector in $\bC^3$, and hence we obtain
\begin{equation}
|\omega(\zeta)|^2 = Z W Z^\ast \:,\: \: \: W=vv^*
\end{equation}
Clearly, \eqref{tacond} implies $T_A = - W$. By construction, W is a positive semidefinite, rank 1 matrix, which proves the "only if" part.

Now suppose $\Rk(T_A)=1$ and $T_A \leq 0$. Since $T_A$ is Hermitian this implies there exists a unitary $3 \times 3$ matrix $U= (u_1 | u_2 |u_3)$ (where the $u_j$ are the columns of $U$), so that:
\begin{equation}
T_A=U \left( \begin{array}{ccc} -\kappa^2 &0&0\\0&0&0\\0&0&0 \end{array} \right) U^\ast
\end{equation}
and, hence,
\begin{equation}
ZT_A Z^\ast = - |\kappa Zu_1|^2.
\end{equation}
Thus, equation \eqref{tacond} is satisfied with
\[
\omega(\zeta)= \kappa Zu_1.
\]
\end{proof}

For fixed $a\in \bR$, $b\in \bC$ (with $ab\neq 0$), we shall now determine for which values of $\tau, \rho,\in \bR$ and $\sigma\in \bC$ the matrix $T_A$ in \eqref{ta} has rank 1 and is negative semidefinite.  We will break the possibilities for $a,b$ up into cases. \medskip

\noindent
{\bf Case 1: $a\neq0$, $b=0$.}
To obtain $\Rk T_A=1$, we clearly must have $\sigma=0$ and we are left with:
\begin{equation}
T_A = \left(
\begin{array}{ccc} \tau + a&0 & 0\\
0&\tau+\rho -4a&0\\
0&0&\rho +a
\end{array} \right)
\end{equation}
Now, $\Rk T_A=1$ can only be achieved with $\tau=-a, \rho=-a$ or $\tau=-a, \rho=5a$ or $\tau=5a, \rho=-a$. This leaves three possibilities for $T_A$:
\begin{equation}
T_A \in \left \{
\left(
\begin{array}{ccc} 0&0 & 0\\
0& -6a&0\\
0&0&0
\end{array} \right)
,
\left(
\begin{array}{ccc} 0&0 & 0\\
0& 0&0\\
0&0&6a
\end{array} \right)
,
\left(
\begin{array}{ccc} 6a&0 & 0\\
0& 0&0\\
0&0&0
\end{array} \right)
\right \}
\end{equation}
We conclude that there is at most one choice of $\tau,\rho,\sigma$ such that $T_A$ has rank 1 and is negative semidefinite if $a>0$, and at most two choices if $a<0$.
\medskip

\noindent
{\bf Case 2: $b\neq 0$.}
A moments reflection reveals that if $b\neq 0$ and $\Rk T_A=1$, then either $\sigma=b$ or $\sigma=-b$. If $\sigma=b$, then we must also have $\rho=-a$ and thus
\begin{equation}
T_A = \left(
\begin{array}{ccc} \tau + a&2b & 0\\
2\bar{b}&\tau-5a&0\\
0&0&0
\end{array} \right).
\end{equation}
The rank 1 condition for $T_A$ then implies $\tau^2-4a\tau-5a^2-4|b|^2 = 0$ or, equivalently, $\tau_{\pm}=2a\pm\sqrt{9a^2+4|b|^2}$. Let $\lambda=\lambda_{\pm}$ denote the non-zero eigenvalue of $T_A$ with $\rho=-a$ and $\tau=\tau_{\pm}$. A straightforward calculation shows that $\lambda_\pm = \mp \sqrt{9a^2+4|b|^2}$ and, hence, with $\sigma =b$ there is only one choice of $\tau, \rho$ for which $T_A$ has rank 1 and is negative semidefinite.

The choice $\sigma=-b$ is symmetric to the one where $\sigma=b$, with the roles of $\tau$ and $\rho$ reversed. Again we conclude that with $\sigma=-b$ there is only one choice for $\tau, \rho$ for which $T_A$ has rank 1 and is negative semidefinite.

Thus, when $b\neq0$ we conclude that there are precisely two choices of $\tau, \rho, \sigma$ for which $T_A$ has rank 1 and is negative semidefinite.
\medskip

We can summarize the analysis above in the following proposition.

\begin{prop}\label{Gauss-prop} Let $(\theta,\theta^\alpha)$ be an admissible CR coframe for $M\subset \bC^3$ such that the sectional CR curvature $S(\zeta,\bar \zeta)$ has the normalized form \eqref{abccond}. Consider the Gauss Equation for the second fundamental forms $\omega(\zeta)$ of CR mappings $f\colon M\subset \bS^7\subset\bC^4$,
\begin{equation}
S(\zeta,\bar{\zeta}) + \zeta A \zeta^\ast \db\zeta\db^2 = -|\omega(\zeta)|^2,
\label{Acond57}
\end{equation}
where $A$ is any Hermitian $2\times2$ matrix. Let $L_S\colon X\to X$ be the linear map, given by \eqref{L-S}, associated to the CR curvature. The following holds:
\medskip

\noindent{\rm{\bf (i)}} If $\Rk L_S<2$, then $S(\zeta,\bar\zeta)\equiv 0$, and $\omega(\zeta)\equiv0$ is the only solution to \eqref{Acond57}.
\smallskip

\noindent{\rm{\bf (ii)}} If $\Rk L=2$ and $\tr L_S <0$, then there is precisely one solution $\omega(\zeta)$ to \eqref{Acond57}, up to the equivalence given by \eqref{SFFeq}.
\smallskip

\noindent{\rm{\bf (iii)}} If $\Rk L_S=3$ or $\Rk L=2$ and $\tr L_S \geq 0$, then there are precisely two solutions $\omega(\zeta)$ to \eqref{Acond57}, up to the equivalence given by \eqref{SFFeq}.
\end{prop}

\section{Proof of Theorem \ref{Main1}}

Suppose  $f,\tilde{f}:M\subset \mathbb{C}^3 \rightarrow \mathbb{S}^7 \subset \mathbb{C}^3$ are two non-constant CR mappings. As mentioned already in the introduction, it is well known (using the Hopf boundary point lemma and an identity relating the Levi forms of $M$ and $\bS^7$) that both $f$ and $\tilde f$ are transversal, local embeddings of $M$ into the sphere $\bS^7$. By the unique continuation property of CR functions along the connected, strictly pseudoconvex hypersurface $M$ (any CR function extends smoothly and holomorphically to a connected, open subset of $\bC^3$ with $M$ in its boundary), it suffices to show that there exists an automorphism $T$ of $\bS^7$ such that $f=T\circ \tilde f$ locally near some point $p\in M$ to conclude that $f=T\circ \tilde f$ holds on $M$. We shall pick a point $p\in U\subset M$ and an open neighborhood $V\subset U$ of $p$ such that the rank of $L_S$ is constant in $V$. Moreover, if this rank is equal to 2, then we shall further require that either $\tr L_S\equiv 0$ in $V$, or $\tr L_S$ has constant sign in $V$.

Let us choose an admissible CR coframe $( \theta, \theta^\alpha)$ for $M$ in $V$ such that the CR curvature has the normalized form \eqref{abccond}.
Let $(\theta, \theta^A )$ and $(\tilde{\theta}, \tilde{\theta}^A)$ be admissible coframes on $\mathbb{S}^7$ defined on open neighborhoods of $f(p)$ and $\tilde{f}(p)$ and adapted to $f$ and $\tilde{f}$, respectively, with respect to the chosen coframe $(\theta, \theta^\alpha)$ on $M$. (Here, we may need to shrink $V$.) Assume now that the respective second fundamental forms $\omega(\zeta)$ and $\tilde{\omega}(\zeta)$ of $f$ and $\tilde f$ at $q$, for every $q\in V$, satisfy
\begin{equation}\label{SFFeqh}
\tilde{\omega}(\zeta) \equiv h(q) \omega(\zeta),
\end{equation}
where $h\colon V \to \bC$ is some unimodular function. (We remark here that the dependence on $q$ in the second fundamental forms $\omega(\zeta)$ and $\tilde{\omega}(\zeta)$ has been suppressed in the notation.) If we make the following unitary change of admissible coframe adapted to $f$, $\theta^3\mapsto h\theta^3$ (or, equivalently, for the dual frame $L_3\mapsto L^3/h$), then relative to this new coframe the second fundamental form of $f$ satisfies
$$
(\omega_{\alpha \: \: \: \beta}^{\:\:\:a})=(\tilde{\omega}_{\alpha \: \: \: \beta}^{\:\:\:a}).
$$
We note that in the situation considered here ($N-n=1$) the second fundamental form $(\omega_{\alpha}{}^a{}_{\beta})$ is nondegenerate precisely when it is not identically zero. Let us proceed under the additional assumption that $(\omega_{\alpha}{}^a{}_{\beta})$ is not zero in $V$. We may apply Theorem \ref{ehz61} (a.k.a.\ Theorem 6.1 from \cite{EHZ04}) to conclude
\begin{equation}
\tilde{f}^{\ast}(\tilde{\phi}_B^{\: \: \: A}, \tilde{\phi}^{ A}, \tilde{\psi}) = f^{\ast} (\phi_B^{\: \: \: A}, \phi^{A}, \psi),
\label{pullbacks}
\end{equation}
where these forms are the Chern-Moser forms first pulled back to $\hat{M}$ via the choices of adapted coframes on $\hat{M}$ near $f(p)$ and $\tilde{f}(p)$, respectively, and then pulled back to $M$ via $f$ and $\tilde f$. We now proceed exactly as in Section 8 of \cite{EHZ04} (see also Lemma (1.1) in \cite{Webster79}) to conclude that there is an automorphism of the sphere, $T\in \Aut(\mathbb{S}^7)$, such that
\begin{equation}\label{fequal}
f=T\circ \tilde{f}
\end{equation}
in $V$ and, hence, on $M$. Thus, given $f$ and $\tilde f$ as above, the identity \eqref{SFFeqh} implies \eqref{fequal}, provided $(\omega_{\alpha}{}^a{}_{\beta})$ is not zero in $V$.

We shall now complete the proof of Theorem \ref{Main1}. The proof of part (ii) in the case where $\Rk L_S<2$ is given in the paragraph following Lemma \ref{Fischer} above. If $\Rk L_S=2$ and $\tr L_S<0$ in $V$, then Proposition \ref{Gauss-prop} implies that the second fundamental forms of $f_1$ and $f_2$, with $f_1=f$ and $f_2=\tilde f$, satisfy \eqref{SFFeqh} and do not vanish in $V$. Hence, the conclusion of Theorem \ref{Main} (ii) follows also in this case (by the arguments in the preceding paragraph). Now, to prove part (i) of Theorem \ref{Main1}, it clearly suffices to assume that $\Rk L_S\geq 2$ (since part (ii) has been proved). Given $f_1$, $f_2$, and $f_3$ as in part (i), Lemma \ref{Fischer} implies that there are $j\neq k$ with $j,k\in\{1,2,3\}$ such that the second fundamental forms of $f_j$ and $f_k$, with $f_j=f$ and $f_k=\tilde f$, satisfy \eqref{SFFeqh} and do not vanish in $V$. Again, the conclusion of Theorem \ref{Main1} (i) follows as above.
\qed

\section{Examples and Concluding Remarks}\label{Examples}

As mentioned in the introduction, there are well known examples showing that Corollary \ref{Main0} is sharp in the case where $M$ is the sphere (or spherical, $\mu(M)=0$), namely the linear embedding and the Whitney map. When $M$ is not spherical, however, constructing inequivalent mappings in the critical codimension where rigidity potentially breaks down has turned out to be quite elusive so far, and the authors do not have any examples in the setting of Theorem \ref{Main} where there are two inequivalent mappings. (Of course, by using inequivalent mappings between spheres, it is easy to construct inequivalent mappings for general $M$ when the codimension is sufficiently high.) We conclude this paper by recording a couple of examples of CR embeddings $M\subset\bC^3\to \bS^7$ where $\mu(M)=1$. These examples were both originally given by Webster in \cite{Webster78b}.

\begin{ex} {\rm Suppose $M\subset \bC^3$ is a real ellipsoid given by the equation: \[M = \{ z \in \mathbb{C}^3 \colon \db z\db^2 + b(z)+\overline{b(z)}-1 = 0 \}\] where $b(z)= b_{\alpha \beta} z^\alpha z^\beta$ is a homogeneous degree two holomorphic polynomial in $z$. Then it is easy to verify that the map $f:M\to \mathbb{S}^7  \subset \mathbb{C}^4$ given by
\begin{equation}
f(z_1,z_2,z_3) = \frac{1}{1-b(z)} (z_1,z_2,z_3,b(z))
\end{equation}
sends $M$ into the sphere $\mathbb{S}^7$.
}
\end{ex}

\begin{ex} {\rm  More generally, any real hypersurface $M \subset \mathbb{C}^n$ with a defining equation of the form \[M = \{ z \in \mathbb{C}^n \colon |z|^2 + b(z) + \overline{b(z)} -1 \equiv 0 \}, \] with $b(z)$ holomorphic in $\bC^n$, may be embedded into the sphere $\mathbb{S}^{2n+1} \subset \mathbb{C}^{n+1}$ via the map:
\[
f(z) = \frac{1}{1-b(z)} (z_1,z_2,z_3,...,z_n,b(z))
\]
}
\end{ex}

\def\cprime{$'$}

\end{document}